%% file: main.tex
\documentclass[a4paper,12pt,oneside,headsepline]{scrartcl}

\input{packages}
\input{layout}

\input{shortcuts}

\input{shortcuts_this}
\DeclareMathAlphabet{\mathcal}{OMS}{cmsy}{m}{n}

\newcommand{\headertitle}{{\normalfont%
}}
\newcommand{\headerauthors}{%
}

\title{%
  Relating depth graded and block graded motivic Lie algebras
}
\author{%
  Adam Keilthy%
}

\begin{document}

\thispagestyle{scrplain}
\begingroup
\deffootnote[1em]{1.5em}{1em}{\thefootnotemark}
\maketitle
\endgroup

\begin{abstract}
    %The algebra of multiple zeta values contains a number of somehow ``unexpected'' relations, leading to a ``depth defect'' that does not exist for the algebra of Euler sums or cyclotomic multiple zeta values of higher level. For these more general settings, the depth filtration is equal to the coradical filtration coming from the motivic structure, but in the case of multiple zeta values, there is a discrepancy. 
    Using the block filtration as a realisation of the coradical filtration, we study the discrepancy between the depth filtration and the coradical filtration for motivic multiple zeta values. We construct an explicit dictionary between a certain subspace of block graded multiple zeta values and totally odd multiple zeta values and show that all expected relations in the depth graded motivic Lie algebra may be realised in the block graded Lie algebra as the kernel of an explicit map. We also discuss some connections to the uneven Broadhurst-Kreimer conjecture, and outline a possible approach.
\end{abstract}
\section{Introduction}
For a tuple $(k_1,\ldots,k_r)$ of positive integers, with $k_r\geq 2$, one can define the associated multiple zeta value (MZV) as the multiple sum
\[\sum_{0<n_1<\cdots<n_r}\frac{1}{n_1^{k_1}n_2^{k_2}\ldots n_r^{k_r}}.\]
We say this multiple zeta has depth $r$ and weight $k_1+\cdots+k_r$.

The study of multiple zeta values dates back to Euler, but they have recently regained substantial interest due to their appearances across a number of areas of mathematics and theoretical physics, including the study of Lie algebras \cite{SchnepsPoisson06}, knot theory \cite{BarNatan98} and string amplitudes \cite{BroadhurstKConj96}. They form an algebra under the harmonic, or stuffle, product:
\begin{example}
\begin{align*}
\zeta(k)\zeta(l) =&{} \sum_{m,n>0} \frac{1}{m^kn^l}\\
=&{} \sum_{0<m<n}\frac{1}{m^kn^l} + \sum_{0<n<m}\frac{1}{n^lm^k}+\sum_{0<m=n}\frac{1}{n^{k+l}}\\
=&{} \zeta(k,l)+\zeta(l,k)+\zeta(k+l),
\end{align*}
\end{example}
and are known to satisfy many algebraic relations, including the double shuffle relations \cite{Racinet02DoubleShuff}, the associator relations \cite{LeMurakamiMZVAssoc}, the confluence relations \cite{HSConfluence19}, and many more. A long standing question in the field has been to describe a complete set of relations. While there are many candidates, providing such a description is very challenging, as it would encapsulate conjectures such as the algebraic independence of 
\[\{\zeta(3),\zeta(5),\zeta(7),\ldots\}.\]

Multiple zeta values notably arise as periods of mixed Tate motives \cite{DG05MixedTate,BrownMTM12}. As such, it is possible to upgrade them to objects called motivic periods, the algebraic theory of which is much more rigid. As motivic periods, the algebra $\cH$ of motivic multiple zeta values (mMZVs) is graded by weight, and the dimensions of the weight graded pieces are encoded in the generating series
\[\frac{1-s^2}{1-s^2-s^3}.\]
The coefficients of this series provide an upper bound on the dimension as a consequence of general motivic theory \cite{Goncharov01MPL,Terasoma01MTM}. The proof that this  describes exactly the dimensions is quite involved, and relies on the introduction of Brown's level filtration \cite{BrownMTM12} on a particular subspace of mMZVs. This is a sharp contrast to similar results due to Deligne \cite{Del10-23468}, in which a comparatively simpler proof is possible by considering the associated graded algebra with respect to the filtration induced by depth.

The failure of depth in the case of mMZVs is due to the existence of certain ``exceptional'' relations arising from connections to cusp forms \cite{GKZ06,Pollack09Thesis,BrownDepthGraded21}. For example, the unique cusp form of weight 12 corresponds to the following ``unexpected'' depth drop
\[28\zeta(3,9) + 150\zeta(5,7) +168\zeta(7,5) = \frac{5197}{691}\zeta(12).\]
These exceptional relations make the study of depth graded MZVs challenging, as their origin is still somewhat mysterious \cite{Hain2016hodge,Hain2020universal}.

In contrast, the block filtration, introduced by the author \cite{KeilthyBlock21,KeilthyThesis20} based on work by Charlton \cite{CharltonThesis,CharltonBlock21}, has no such defects in in associated graded. Indeed, we obtain an essentially isomorphic algebra, in which a complete description of (motivic) relations is possible in low block degree. However, as established in a recent article \cite{CK22Evaluation242}, traces of the depth defect may still be seen in the block graded algebra, at least modulo products. The cusp form relation above, becomes the relation
\[7\zeta(2,2,2,4,2) +\frac{5}{2}\zeta(2,2,4,2,2) + 7\zeta(2,4,2,2,2) = 0 \text{ modulo products}.\]
In \cite{CK22Evaluation242}, we establish that such cuspidal relations are a consequence of the block graded relations established in \cite{KeilthyBlock21}.

As such, one might hope to find a natural interpretation of the depth defect reflected in the block graded algebra. In this article, we make first steps towards such a connection. Specifically, we construct a commutative triangle
\[\begin{tikzcd}
    \bgmot\arrow[d]\arrow[rd] & {} \\
    \sgmot\arrow[r] & \egmot
\end{tikzcd}\]
involving the Lie algebra $\bgmot$ dual to block graded MZVs, a Lie subalgebra $\sgmot$ of the Lie algebra dual to depth graded MZVs, and explicit maps to a common quotient Lie algebra $\egmot$. Based on a standard conjecture, we conjecture that $\sgmot\cong\egmot$, thereby giving us a way to theoretically study the subalgebra $\sgmot$ via the block graded Lie algebra $\bgmot$. In particular, we can see ``exceptional'' relations in the depth graded setting in the kernel of the map $\bgmot\to\egmot$.

Dualising this map, we obtain via \cref{thm:evenfunctionalsgiverelations} a dictionary between a subspace of MZVs of depth $r$ and MZVs of block degree $r$, and can give an explicit representation of the restriction of this map to totally odd MZVs
\[\zeta(2k_1+1,\ldots,2k_r+1)\]
via ``almost'' Hoffman elements of the form
\[\zeta_1(\{2\}^{a_1},3,\ldots,3,\{2\}^{a_r})\]
where $\zeta_1(\ldots)$ is a certain regularized iterated integral and $\{m\}^n$ represents the sequence consisting of $m$ repeated $n$ times.

Assuming a result called the uneven Broadhurst-Kreimer conjecture, this dictionary then represents an explicit representation of the natural map from depth graded MZVs to block graded MZVs coming from the inclusion of depth as a subfiltration of the block filtration.

\begin{example}
    Some examples of applications of \cref{thm:evenfunctionalsgiverelations} include the following relations modulo products and terms of lower block degree:
    \begin{align*}
        \zeta(2,2,5,2)-\zeta(2,5,2,2) =&{} 8\zeta(1,3,7) -8\zeta(1,5,5),\\
        \zeta(2,3,3,2,2,2)-\zeta(2,2,2,3,3,2) =& 35\zeta(6,8) +28\zeta(5,9) + 15\zeta(7,7),\\
        8\zeta(3,3,5) =&{} \zeta(2,3,4,2) - \zeta(2,1,1,2,3,2)-2\zeta(1,2,3,3,2)\\
        &-2\zeta(2,1,3,3,2)-2\zeta(2,3,1,3,2)-2\zeta(2,3,3,1,2).
    \end{align*}
\end{example}
The structure of the article is as follows.

In \cref{sec:mmzv}, we will give a quick guide to the necessary background. As the theory of motivic periods is deep and extensive, we will give only a bare bones explanation. We will introduce a formal notion of motivic multiple zeta values and their dual Lie algebras. We will then define both the depth and block filtrations and give polynomial expressions for the associated graded Lie algebras.

In \cref{sec:relatingDandB}, we will establish an explicit connection between the depth graded and block graded Lie algebras via a commutative diagram. We use this to connect relations in the depth graded Lie algebra to a subalgebra of the block graded Lie algebra. We also use this commutative diagram to provide an explicit recipe for relating certain block-homogeneous linear combinations of MZVs with depth-homogeneous linear combinations of MZVs. We will also, in particular, give an expression for totally odd MZVs of depth $r$, i.e. those of the form $\zeta(2k_1+1,\ldots,2k_r+1)$ in terms of particular MZVs of block degree $r$, which we will call ``almost'' Hoffman zeta values. 

In \cref{sec:existingconj}, we follow up on a number of connections with existing conjectures about the structure of the depth graded Lie algebra and the space of totally odd MZVs, relating them to a number of conjectures made throughout the article and discussing some potential strategies for approaching these conjectures from a block graded perspective.

\subsubsection*{Acknowledgements}
The author would like to thank Steven Charlton for several valuable discussions on topics addressed in this article, and Martin Raum for his advice and comments. The author was supported by the “Postdoctoral program in Math-
ematics for researchers from outside Sweden” (project KAW 2020.0254).

\section{Multiple zeta values and motivic periods}\label{sec:mmzv}
Via their representation as iterated integrals
\[\zeta(k_1,\ldots,k_r)=(-1)^r\int_{0\leq t_1\leq \cdots \leq t_{k_1+\cdots+k_r}}\prod_{i=1}^{k_1+\cdots+k_r} \frac{dt_i}{t_i-a_i},\]
where $a_i=1$ if $i\in\{1,k_1+1,k_1+k_2+1,\ldots,k_1+\cdots+k_{r-1}+1\}$ and $a_i=0$ otherwise, multiple zeta values fall into a class of numbers called periods \cite{KZ01}. As such, we are able to lift them to objects called motivic periods - formal algebraic objects satisfying only relations coming from geometry. These formal objects are much simpler to study, but still provide valuable insight to the algebraic structure of multiple zeta values. 

In particular, by studying the algebra of motivic multiple zeta values, we obtain an algebra that is graded by weight, that comes equipped with a coaction encoding relations among elements of the algebra, and which is known to be (non-canonically) isomorphic to the shuffle algebra
\[\bbQ\langle f_3,f_5,f_7,\ldots\rangle[f_2]\]
with one commutative generator in weight 2, and one non-commutative generator in every odd weight greater than one. In this formal algebraic setting, questions of transcendence are easy to resolve, and so we can consider further simplifications, such as working modulo products, or in an associated graded algebra. The general theory of motivic periods is extensive \cite{BrownMotivicPeriodNotes}, and so we sketch here the essentials, rather than exploring the full connection between motivic multiple zeta values and the Tannakian formalism for the category of mixed Tate motives over $\Spec\bbZ$ and the motivic fundamental group of $\bbP^1\setminus\{0,1\infty\}$ \cite{DG05MixedTate}.

\subsection{A poor man's guide to motivic multiple zeta values and iterated integrals}
\begin{definition}\label{def:Hn}
The algebra $\cH$ of motivic multiple zeta values is the $\bbQ$-algebra spanned by symbols 
\[\I^\frakm(a_0;a_1,\ldots,a_n;a_{n+1})\text{ where }a_i\in\{0,1\}\,,\]
called motivic multiple zeta values or motivic iterated integrals, satisfying the following properties:
\begin{enumerate}
    \item (Equal boundaries) $\I^\frakm(a_0;a_1,\ldots,a_n;a_0)=\delta_{n,0}$,
    \item (Reversal of paths) $\I^\frakm(a_0;a_1,\ldots,a_n;a_{n+1})=(-1)^n\I^\frakm(a_{n+1};a_n,\ldots,a_1;a_0)$,
    \item (Functoriality) $\I^\frakm(a_0;a_1,\ldots,a_n;a_{n+1})=\I^\frakm(1-a_0;1-a_1,\ldots,1-a_n;1-a_{n+1})$,
    \item (Shuffle product) For $1<r<n$, denote by $\Sh_{r,n-r}$ the set of permutations $\sigma$ on $n$ satisfying
    \[ \sigma(1)<\sigma(2)<\cdots<\sigma_r\text{ and }\sigma(r+1)<\cdots<\sigma_n\,.\]
    Then
    \[\I^\frakm(0;a_1,\ldots,a_r;1)\I^\frakm(0;a_{r+1},\ldots,a_n;1)=\sum_{\sigma\in\Sh_{n,r}}\I^\frakm(0;a_{\sigma^{-1}(1)},\ldots,a_{\sigma^{-1}(n)};1)\,,\]
    \item (Period map) There is a ring homomorphism $\per:(\cH,\sh)\to (\bbC,{}\cdot{})$, called the period map, sending a motivic iterated integral to the corresponding complex iterated integral,
    \item (Relations) Relations among elements in $\cH$ are precisely those elements in $\ker(\per)$ that are \textit{motivically stable} - we will make this precise in a moment.
\end{enumerate}

For a tuple of positive integers $(k_1,\ldots,k_d)$, and \( \ell \geq 0 \), we write \( \zeta^\frakm = \zeta^\frakm_0 \) and  
\begin{equation}\label{eqn:zasint}
\zeta^\frakm_\ell(k_1,\ldots,k_d)\coloneqq(-1)^d\I^\frakm(0;\{0\}^\ell,1,\{0\}^{k_1-1},\ldots,1,\{0\}^{k_d-1};1)\,,
\end{equation}
where $\{0\}^n$ denotes $n$ repeated zeroes.
\end{definition}

\begin{remark}
    The algebra $\cH$ is graded by weight. As such, we will denote the weight $N$ component of $\cH$ and use similar notation for all spaces derived from it. We will also use the notation $\cH_{>0}$ to denote the vector subspace of $\cH$ spanned by mMZVs of positive weight, and similarly for all spaces derived from $\cH$.
\end{remark}

\begin{remark}\label{rem:regularisation}
A priori, the period map of \cref{def:Hn} will take some motivic iterated integrals to divergent integrals. However, there is a unique ring homomorphism $\per_{\sh,T}:(\cH,\sh)\to(\bbC[T],\cdot)$ taking a motivic iterated integral to the corresponding complex iterated integral when it converges, and taking $\I^\frakm(0;1;1)$ to $T$ \cite{ChenIteratedIntegrals77}. This is called shuffle regularisation, and we will always take $T=0$. There is an analogous notion of stuffle regularisation \cite{IKZ06}, but we will not make use of it here.
\end{remark}

%\begin{remark}
%The reversal of paths property and the functoriality property give an important relation for mMZVs called the \emph{duality} relation:
%\[\I^\frakm(0;a_1,\ldots,a_n;1)=(-1)^n\I^\frakm(0;1-a_n,\ldots,1-a_1;1)\]
%\end{remark}

Denote by $\cA:=\cH/(\zeta^\frakm(2))$ and denote the image of $\I^\frakm(a_0;a_1,\ldots,a_n;a_{n+1})$ in this quotient by $\I^\frakA(a_0;a_1,\ldots,a_n;a_{n+1})$. The Tannakian formalism tells us that $\cA$ coacts on $\cH$, with an explicit formula due to Goncharov \cite{GoncharovGalois01}:
\begin{align*}\label{eqn:coaction}
\Delta:\cH \to \cA &\otimes\cH,\\
\I^\frakm(a_0;a_1,\ldots,a_n;a_{n+1})&\mapsto\\
\sum_{\substack{i_0<i_1<\cdots<i_{k+1}\\i_0=0,\, i_{k+1}=n+1}}\left(\prod_{s=0}^k\I^\fraka(a_{i_s};a_{i_s+1},\ldots,a_{i_{s+1}-1};a_{i_{s+1}})\right)&\otimes\I^\frakm(a_0;a_{i_1},\ldots,a_{i_k};a_{n+1}))\,.\end{align*}
This is called the motivic coaction, and we call an element $R\in\ker(\per)$ motivically stable if $\per(R^\prime)=0$ for all conjugates $R^\prime$ of $R$ under the coaction.

It is a standard conjecture that the period map is injective. Indeed, most known relations among multiple zeta values are known to be motivically stable, and hence define relations among motivic multiple zeta values. Indeed, motivic relations include the double shuffle relations, the associator relations, and the confluence relations.

\subsubsection{Depth and block filtrations}\label{subsec:filtrations}
From this point, we shall consider mMZVs modulo $\zeta^\frakm(2)$ and work exclusively in $\cA$, though much of this can be easily modified for $\cH$. The formula for the motivic coaction induces a coproduct on $\cA$, giving $\cA$ the structure of a Hopf algebra. As in the case of multiple zeta values, we can assign to a motivic iterated integral $\I^\fraka(0;a_1,\ldots,a_n;1)$ a weight of $n$ and a depth given by $\#\{a_i\mid a_i=1\}$.

Weight defines a Hopf algebra grading on $\cA$, while depth only induces a Hopf algebra filtration
\[\cD_n\cA := \langle \zeta^\fraka(k_1,\ldots,k_r)\mid r\leq n\rangle_\bbQ.\]

The Hopf algebra structure on $\cA$ naturally induces another filtration, called the coradical filtration. Defining $\cC_0\cA$ to be the sum of all simple subcoalgebras (in $\cA$, this is just a copy of $\bbQ$), we define 
\[\cC_n\cA := \Delta^{-1}(\cA\otimes\cC_0\cA + \cC_{n-1}\cA\otimes\cA).\]
In general, the coradical filtration is a very natural filtration to consider, with connections to Hochschild cohomology \cite{Stefan98CoradicalCohomology}, connections with injectivity of coalgebra morphisms, and immediate compatibility with most coalgebraic constructions. In our case, we in fact have that the coradical filtration defines a grading on $\cA$, expressed in terms of $f$-degree for any choice of coalgebra isomorphism
\[\cA\cong \bbQ\langle f_3,f_5,f_7\ldots\rangle.\]

Usually, the coradical filtration can be quite hard to describe in general. Indeed, computing the $f$-degree of any multiple zeta value is quite difficult. However, using the notion of an alternating block decomposition, first introduced by Charlton \cite{CharltonBlock21,CharltonThesis}, we can define a combinatorial degree function for which the associated filtration equals the coradical filtration \cite{KeilthyBlock21}

\begin{definition}
    We call a word $w$ in two letters $\{x,y\}$ an alternating block if it is of the form $xyxy\ldots$ or $yxyxy\ldots$. The alternating block decomposition of a word $w$ is the unique minimal factorisation into alternating blocks.
\end{definition}

\begin{definition}
    The block degree $\degb(w)$ of a word $w$ in two letters $\{x,y\}$ is equal to one less than the number of blocks in its alternating block decomposition. Equivalently, the block degree $\degb(w)$ is equal to the number of occurrences of a subword $xx$ or $yy$ in $w$. 
\end{definition}

\begin{definition}
    The block filtration on $\cA$ is the increasing filtration
    \[\cB_n\cA:=\langle \I^\fraka(0;w;1)\mid \degb(0w1)\leq n\rangle_\bbQ.\]
\end{definition}

\begin{definition}
    Given a tuple $(\ell_0,\ell_1,\ldots,\ell_n)$ of positive integers, define
    \[\I^\fraka(\ell_0,\ldots,\ell_n):= \I^\fraka(0;w;1)\]
    where $w$ is the unique word in $\{0,1\}$ such that the alternating block decomposition of $0w1$ consists of alternating blocks of lengths $(\ell_0,\ldots,\ell_n)$.
\end{definition}

With this definition, the block filtration is given by
\[\cB_n\cA:=\langle \I^\fraka(\ell_0,\ldots,\ell_r)\mid r\leq n\rangle_\bbQ.\]

\begin{proposition}\label{prop:blockiscoradical}
    The block filtration on $\cA$ is equal to the coradical filtration
    \[\cB_n\cA=\cC_n\cA.\]
\end{proposition}
For the proof of the above, along with further properties of the block filtration, we refer the reader to the author's previous work \cite{KeilthyBlock21}.

\begin{corollary}\label{cor:depthsubblock}
    The depth filtration is a subfiltration of the block filtration
    \[\cD_n\cA\subset\cB_n\cA.\]
\end{corollary}
\begin{proof}
    Direct computation shows that that $\cD_1\cA\subset\cB_1\cA$. Furthermore, from the formula for the coproduct, it is easy to see that
    \[\Delta\cD_n\cA \subset \cA\otimes \bbQ + \cD_{n-1}\cA\otimes\cA, \]
    and so
    \[\cD_n\cA\subset\Delta^{-1}(\cA\otimes\bbQ + \cD_{n-1}\cA\otimes\cA).\]
    BY induction on $n$, we may conclude that
    \[\cD_n\cA\subset\Delta^{-1}(\cA\otimes\bbQ+\cD_{n-1}\cA\otimes\cA)\subset\Delta(\cA\otimes\bbQ+\cB_{n-1}\cA)=\cB_n\cA.\]
\end{proof}

\begin{remark}
    As mentioned previously, via the Hopf algebra isomorphism 
    \[\cA\cong \bbQ\langle f_3,f_5,\ldots\rangle,\] the algebra $\cA$ is coradically graded, with grading given by the degree in the $f$-alphabet. However, as computing this isomorphism is in general difficult, the block degree remains the most efficient way of accessing the coradical filtration.
\end{remark}

\subsection{The motivic Lie algebra and its graded counterparts}\label{subsec:Liealgebra}
To simplify further the study of $\cA$ and its defining relations, we can consider the Lie coalgebra of indecomposables 
\[\cL:=\cA_{>0}/\cA_{>0}^2,\] and its dual Lie algebra $\gmot$, called the motivic Lie algebra \cite{DG05MixedTate, BrownMTM12}. This is a Lie subalgebra of $\bbQ\langle e_0,e_1\rangle$ equipped with the Ihara bracket $\{-,-\}$, and encodes relations among mMZVs modulo products as follows: a linear combination $\xi$ of mMZVs vanishes in $\cL$ (i.e. modulo products) if and only if $(\xi,\sigma)=0$ for every $\sigma\in\gmot$. This can be made more explicit by noting that the pairing
\[\cL\times \gmot\to \bbQ\]
is induced by the naive pairing
\begin{align*}
    \bbQ\langle e_0,e_1\rangle \times &\bbQ\langle e_0,e_1\rangle \to \bbQ,\\
    (u,v) &\to\begin{cases}
        1 \text{ if }u=v,\\
        0 \text{ otherwise}
    \end{cases}
\end{align*}
where we identify a word $e_{a_1}e_{a_2}\ldots e_{a_n}$ in the first factor with the corresponding motivic iterated integral $\I^\frakl(0;a_1,\ldots,a_n;1)$ in $\cL$.

As such, in order to describe relations among mMZVs modulo products, it suffices to describe $\gmot$ as a subspace of $\bbQ\langle e_0,e_1\rangle$. From the theory of mixed Tate motives \cite{DG05MixedTate}, we know that $\gmot$ is non-canonically isomorphic to the free Lie algebra
\[\Lie[\sigma_3,\sigma_5,\sigma_7,\ldots]\]
with one generator in every odd weight greater than 1. However, we do not have a canonical representative of each $\sigma_{2k+1}$ for $k\geq 5$, limiting the use we can make of this isomorphism for explicit computations.

\begin{remark}
    While we do not have a canonical representative of each $\sigma_{2k+1}$, we do know that $(\I^\frakl(0;1,\{0\}^{2k};1),\sigma_{2k+1})\neq 0$, and as such we can and will assume $(\I^\frakl(0;1,\{0\}^{2k};1),\sigma_{2k+1})=1$ throughout
\end{remark}

\subsubsection{The depth graded motivic Lie algebra}\label{subsec:dgLie}
The depth filtration on $\cA$ and $\cL$ may be viewed as dual to a decreasing filtration on $\bbQ\langle e_0,e_1\rangle$, also called the depth filtration:
\[\cD^n:=\langle w\in\bbQ\langle e_0,e_1\rangle\mid \deg_{e_1}(w)\geq n\rangle_\bbQ\]
This induces a filtration on the motivic Lie algebra 
\[\cD^n\gmot := \cD^n\cap \gmot\]
that is compatible with the Lie algebra structure. As such we can consider the associated graded Lie algebra.
\begin{definition}
The depth graded Lie algebra is the vector space
\[\dgmot:=\bigoplus_{n\geq 1} \cD^n\gmot/\cD^{n+1}\gmot\]
equipped with the depth graded Ihara bracket $\{-,-\}$.
\end{definition}

This Lie algebra encodes relations among mMZVs modulo products and terms of lower depth \cite{BrownDepthGraded21}. One might hope that considering this depth graded structure might elucidate the structure of $\gmot$, as the depth graded analogues of many relations, such as the double shuffle relations, are significantly simplified. We furthermore have that the image of the non-canonical generators $\sigma_{2k+1}$ in $\cD^1\gmot/\cD^2\gmot$ is independent of our choice of isomorphism
\[\Lie[\sigma_3,\sigma_5,\sigma_7,\ldots]\to \gmot.\]

Explicitly, via the injective map
\begin{align*}
    \cD^n/\cD^{n+1} &\to \bbQ[z_0,z_1,\ldots,z_n],\\
    e_0^{k_0}e_1e_0^{k_1}e_1\ldots e_1e_0^{k_n} &\mapsto z_0^{k_0}z_1^{k_1}\ldots z_n^{k_n},
\end{align*}
the image of $\sigma_{2k+1}$ is given by $(z_0-z_1)^{2k}$. Unfortunately, the depth graded Lie algebra $\dgmot$ is not free: there exist quadratic relations among the images of the $\{\sigma_{2k+1}\}_{k\geq 1}$ related to period polynomials of cusp forms \cite{Pollack09Thesis, BrownDepthGraded21}. The relation associated to the unique cusp form of weight 12 is
\[\{\sigma_3,\sigma_9\} - 3\{\sigma_5,\sigma_7\} = 0 \text{ mod }\cD^3\gmot.\]
As such, there exists a corresponding additional generator in higher depth for each such relation. The structure of $\dgmot$ is conjecturally encoded in the following conjecture, called the homological Broadhurst-Kreimer conjecture \cite{BrownDepthGraded21,BroadhurstKConj96}.

\begin{conjecture}\label{conj:homologicalbk}
Denote by $\,\mathrm{S}$ the space of cusp forms for $\,\mathrm{SL}_2(\bbZ)$. The homology of the depth graded Lie algebra is
\begin{align*}
    H_1(\dgmot) \cong &{} \bigoplus_{k\geq 1}\bbQ\sigma_{2k+1}\oplus \mathrm{S},\\
    H_2(\dgmot) \cong &{} \mathrm{S},\\
    H_i(\dgmot) =&{} 0 \text{ for }i\geq 3.
\end{align*}
\end{conjecture}

However, it is unclear how one might hope to prove this, nor how one could construct the extra ``exceptional'' generators explicitly. Nevertheless, substantial work has been done to understand the polynomial image of $\dgmot$ via the linearised shuffle equations and other symmetries \cite{BrownDepthGraded21, BrownZeta317,Li2019depth,Dietze2017totallyodd}. We shall freely identify $\dgmot$ with it's image in $\bigoplus_{n\geq 1}\bbQ[z_0,\ldots,z_n]$. We will also primarily consider the Lie subalgebra generated by the images of the $\sigma_{2k+1}$.

\begin{definition}
    Denote the Lie subalgebra of \[(\bigoplus_{n\geq 1}\bbQ[z_0,\ldots,z_n],\{-,-\})\] generated by $\{\phi_{2k+1}:=(z_0-z_1)^{2k}\}_{k\geq 1}$ by $\sgmot$.
\end{definition}

\begin{remark}\label{rem:sgiseven}
    This subalgebra $\sgmot$ lies in $\bigoplus_{n\geq 1}\bbQ[z_0,\ldots,z_n]^e$, the subspace of polynomials of even total degree. As a consequence of the depth parity theorem \cite{PanzerParity16,Tsumura2004Parity},  the image of $\dgmot$ lies in this subspace of polynomials of even degree.
\end{remark}

\subsubsection{The block graded Lie algebra}\label{subsec:bgLie}
Analogously to the depth filtration, the block filtration on $\cA$ and $\cL$ may be viewed as dual to a decreasing filtration on $\bbQ\langle e_0,e_1\rangle$, also called the block filtration:
\[\cB^n:=\langle w\in\bbQ\langle e_0,e_1\rangle\mid \degb(e_0we_1)\geq n\rangle_\bbQ\]
This induces a filtration on the motivic Lie algebra 
\[\cB^n\gmot := \cB^n\cap \gmot\]
that is compatible with the Lie algebra structure. As such we can consider the associated graded Lie algebra.
\begin{definition}
The block graded Lie algebra is the vector space
\[\bgmot:=\bigoplus_{n\geq 1} \cB^n\gmot/\cB^{n+1}\gmot\]
equipped with the block graded Ihara bracket $\{-,-\}$.
\end{definition}

This Lie algebra encodes relations among mMZVs modulo products and terms of lower block degree. While many standard relations do not easily descend to the block graded setting, we find many new relations and symmetries appear quite naturally \cite{KeilthyBlock21}. We furthermore have that the image of the non-canonical generators $\sigma_{2k+1}$ in $\cB^1\gmot/\cB^2\gmot$ is independent of our choice of isomorphism
\[\Lie[\sigma_3,\sigma_5,\sigma_7,\ldots]\to \gmot.\]

Explicitly, via the injective map
\begin{align*}
    \cB^n/\cB^{n+1} &\to \bbQ[z_0,z_1,\ldots,z_n],\\
    w &\mapsto z_0^{\ell_0}z_1^{\ell_1}\ldots z_n^{\ell_n},
\end{align*}
where $(\ell_0,\ell_1,\ldots,\ell_n)$ are given by the lengths of the alternating blocks in the alternating block decomposition of $e_0we_1$, the image of $\sigma_{2k+1}$ is given by
\[s_{2k+1}(z_0,z_1)=z_0z_1(z_0-z_1)\left(\frac{(2^{2k+1}-1)(z_0+z_1)^{2k}+(z_0-z_1)^{2k}}{2^{2k}}\right)\]

Unlike the case of the depth graded Lie algebra, the block graded Lie algebra is free \cite{KeilthyBlock21}. Thus we have
\[\gmot\cong\bgmot\]
and so we expect to find all properties of the motivic Lie algebra reflected in the block graded Lie algebra. In particular, we expect to see the depth ``defect'' reflected somewhere in the block graded Lie algebra. The remainder of this article will make this connection more explicit.

In order to do so, it will be convenient to instead consider a particular presentation of $\bgmot$ in polynomials, which we denote by $\rbgmot$, for ``reduced block graded'' Lie algebra. From Lemma 7.4 of \cite{KeilthyBlock21}, the image of $\bgmot$ in $\bbQ[z_0,z_1,\ldots,z_n]$ is divisible by $z_0z_1\ldots z_n(z_0-z_n)$. We define $\rbgmot$ to be the vector subspace of $\bigoplus_{n\geq 1} \bbQ[z_0,z_1,\ldots,z_n]$ obtained from the image of $\bgmot$ divided by $z_0z_1\ldots z_n(z_0-z_n)$ in each degree.

\begin{proposition}[\cite{KeilthyThesis20,KeilthyBlock21}]
    The reduced block graded Lie algebra $\rbgmot$ is the Lie subalgebra of $\bigoplus_{n\geq 1} \bbQ[z_0,z_1,\ldots,z_n]$ generated by
    \[\left\{p_{2k+1}(z_0,z_1)=\frac{(2^{2k+1}-1)(z_0+z_1)^{2k}+(z_0-z_1)^{2k}}{2^{2k}}\mid k\geq 1\right\}\]
    with Lie bracket given by 
    \begin{equation}\label{eq:rbgIhara}
    \{r,q\}(z_0,\ldots,z_n):=\sum_{i=0}^n r(z_i,z_{i+1})\left(q(z_0,\ldots,\hat{z}_{i+1},\ldots,z_n)-q(z_0,\ldots,\hat{z}_i,\ldots,z_n)\right),
    \end{equation}
    for $r(z_0,z_1)\in\bbQ[z_0,z_1]$, $q(z_0,\ldots,z_{n-1})\in\bbQ[z_0,\ldots,z_{n-1}]$, and we consider indices modulo $n$ and $\hat{z}_j$ means we omit $z_j$.
    Furthermore, we have that $\gmot\cong \bgmot\cong \rbgmot$ as Lie algebras.
\end{proposition}

As the above Lie bracket is induced by the Ihara bracket of $\gmot$, we shall refer to it as the Ihara bracket also.

\begin{remark}
    It is clear to see that $\rbgmot$ is a subspace of $\bbQ[z_0,z_1,\ldots,z_n]^e$ of polynomials of even total degree.
\end{remark}

\section{Relating the depth and block graded Lie algebras}\label{sec:relatingDandB}
We will relate the depth and block graded Lie algebras via surjective maps to a subspace of polynomials that are even in every variable. This will dualise to a relationship between depth $r$ totally odd multiple zeta values, i.e. those of the form $\zeta^\frakm(2k_1+1,\ldots,2k_r+1)$, and a certain subspace of motivic iterated integrals of block degree $r$.

\begin{definition}
    Define projection to the totally even part
    \[\pi_e:\bigoplus_{n\geq 1}\bbQ[z_0,z_1,\ldots,z_n]\to \bigoplus_{n\geq 1} \bbQ[z_0^2,z_1^2,\ldots,z_n^2]\]
    to be the projection map onto the part that is even in every variable. Explicitly, for 
    \[f(z_0,\ldots,z_n)\in\bbQ[z_0,\ldots,z_n]\]
    the projection is given by
    \[\pi_e(f)(z_0,\ldots,z_n)=\frac{1}{2^{n+1}}\sum_{I\subset\{0,\ldots,n\}}f((-1)^{\bbone_{0\in I}}z_0,(-1)^{\bbone_{1\in I}}z_1,\ldots,(-1)^{\bbone_{n\in I}}z_n),\]
    where $\bbone_{i\in I}$ is the indicator function for $(i\in I)$.
\end{definition}

\begin{lemma}\label{lem:evenprojectioncommutes}
    The Ihara bracket of $\rbgmot$ defined by \cref{eq:rbgIhara} commutes with projection to the totally even part:
    \[\pi_e(\{r,q\})(z_0,\ldots,z_n) = \{\pi_e(r),\pi_e(q)\}(z_0,\ldots,z_n),\]
    for $r(z_0,z_1)\in\bbQ[z_0,z_1]^e$, and $q(z_0,\ldots,z_{n-1})\in\bbQ[z_0,\ldots,z_{n-1}]^e$.
\end{lemma}
\begin{proof}
    By linearity of the projection map and the structure of the Ihara bracket, it suffices to show that
    \[\pi_e\left(r(z_0,z_1)q(z_1,\ldots,z_n)\right)=\pi_e(r(z_0,z_1))\left(\pi_e(q(z_1,\ldots,z_n))\right).\]
    Since $r(z_0,z_1)$ and $q(z_1,\ldots,z_n)$ are of even total degree, the contribution of $q(z_1,\ldots,z_n)$ to the left hand side must come from the part of $q(z_1,\ldots,z_n)$ that is even in $z_2,\ldots,z_n$ and hence the part that is even in every variable. The equality then follows.
\end{proof}

\begin{definition}
    Denote by $\egmot\subset \bigoplus_{n\geq 1}\bbQ[z_0,\ldots,z_n]^e$ the Lie algebra generated by 
    \[\{\pi_e(p_{2k+1})(z_0,z_1) = (z_0+z_1)^{2k}+(z_0-z_1)^{2k} \mid k\geq 1\}\]
    under the Ihara bracket of $\rbgmot$. Denote by $e_{2k+1}:=\pi_e(p_{2k+1})$.
\end{definition}

\begin{corollary}\label{cor:evenprojhomomorphism}
    Projection to the totally even part defines a Lie algebra homomorphism
    \[\rbgmot\to \egmot\]
\end{corollary}

The above corollary is intimately related to the results of \cite{CK22Evaluation242}. Indeed, one of the major results of the author's previous work may be reformulated as the statement that the kernel of this homomorphism in block degree 2 is equal to the space of quadratic relations of of $\dgmot$. This observation motivates the following considerations.

Recall that we defined $\sgmot$ to be the the Lie subalgebra of $\dgmot$ generated by the images of $\{\sigma_{2k+1}\}_{k\geq 1}$. Brown \cite{BrownDepthGraded21} determines the polynomial representation of the Ihara bracket explicitly to be the antisymmetrization of
\begin{align*}
(f\circ g)(z_0,\ldots,z_{r+s}) &{}= \sum_{i=0}^s f(z_i,\ldots,z_{i+r})g(z_0,\ldots,z_i,z_{i+r+1},\ldots,z_{r+s})\\
+(-1)^{\deg f + r}&{}\sum_{i=1}^s f(z_{i+r},\ldots,z_i)g(z_0,\ldots,z_{i-1},z_{i+r},\ldots,z_{r+s})
\end{align*}
for $f(z_0,\ldots,z_r)\in\bbQ[z_0,\ldots,z_r]$, $g(z_0,\ldots,z_s)\in\bbQ[z_0,\ldots,z_s]$. As noted in \cref{rem:sgiseven}, $\dgmot$ is a subspace of polynomials of even total degree, and so $\deg f\in 2\bbZ$ for our purposes.

In the same article, Brown establishes that elements of $\dgmot$ satisfy certain dihedral symmetries. In particular,
\[f(z_0,\ldots,z_r) = (-1)^{r+1}f(z_r,\ldots,z_0)\]and
\[f(z_0,\ldots,z_r)=f(z_1,\ldots,z_r,z_0)\]
for any $f\in\mathfrak{dg}$. Using this dihedral symmetry, we find that the Lie bracket of $f(z_0,z_1)$ and $g(z_0,\ldots,z_{n-2})$ is given by
\[\{f,g\}(z_0,\ldots,z_{n-1})=\sum_{i=0}^{n-1} f(z_i,z_i+1)\left(g(z_0,\ldots,\hat{y}_{i+1},\ldots,z_{n-1})-g(z_0,\ldots,\hat{y}_i,\ldots,z_{n-1})\right).\]

Note that this is identical in form to the Ihara bracket of $\rbgmot$ \cref{eq:rbgIhara}. Furthermore, the projection $\pi_e$ takes $\phi_{2k+1}(z_0,z_1)$ to $\frac{1}{2}e_{2k+1}(z_0,z_1)$. Thus we may conclude the following.

\begin{proposition}\label{prop:stoeissurj}
    Projection to the totally even part defines a surjective Lie algebra homomorphism
    \[\pi_e:\sgmot\to \egmot.\]
\end{proposition}

In fact, we can more precisely say the following.

\begin{theorem}\label{thm:easycommutativediagram}
Via the standard identification with polynomial algebras, projection onto the totally even part defines surjective maps
\begin{align*}
    \pi_e:\bgmot &\twoheadrightarrow \egmot,\\
    2\pi_e:\sgmot &\twoheadrightarrow \egmot,
\end{align*}
such that the following diagram commutes.
\[\begin{tikzcd}
{} & \Lie [\sigma_3,\sigma_5,\ldots]\arrow{d}{\text{Lowest depth}} \arrow{r}{\simeq} & \bgmot\arrow{d}{\pi_e}\\
\dgmot\arrow{r} & \sgmot\arrow{r}{2\pi_e} & \egmot\\
\end{tikzcd}\]
\end{theorem}

\begin{corollary}\label{cor:relationsamongkernels}
There exists a short exact sequence of Lie algebras
\[\begin{tikzcd}
    0 \arrow{r} & \ker(\gmot\to\sgmot)\arrow{r} & \ker(\bgmot\to\egmot)\arrow{r} & \ker(\sgmot\to\egmot)\arrow{r} & 0.
\end{tikzcd}\]
\end{corollary}
\begin{proof}
    All of the above claims are general properties of a commutative triangle of epimorphisms 
    \[\begin{tikzcd}
        D\arrow{d}\arrow{rd} & {}\\
        G\arrow{r} & E
    \end{tikzcd}\]
    in an abelian category. Specifically, we consider the following commutative diagram
    \[
    \begin{tikzcd}
       {}  & 0\arrow{d} & 0\arrow{d} & 0\arrow{d} & {}\\
        0\arrow{r} & A\arrow{r}\arrow{d} & B\arrow{r}\arrow{d} & 0\arrow{r}\arrow{d} & 0\\
        0\arrow{r} & C\arrow{r}\arrow{d}{h} & D\arrow{r}{f}\arrow{d}{\pi} & E\arrow{r}\arrow{d}{\id} & 0\\
        0\arrow{r} & F\arrow{r}\arrow{d} & G\arrow{r}{g}\arrow{d} & E\arrow{r}\arrow{d} & 0\\
        {} & 0 & 0 & 0 & {}
    \end{tikzcd}
    \]
    where
    \begin{align*}
        B &:= \ker(\pi)\\
        C &:= \ker(f)\\
        F &:= \ker(g)\\
        A &:= \ker(h)
    \end{align*}
where $h$ is the natural map from $\ker(f)\to \ker(g)$ induced by $\pi$. The bottom two rows are exact, and thus the four lemma implies $h$ is an epimorphism. Hence the arrow $A\to B$ is well defined, and the columns are exact. Thus, by the nine lemma, the top row is exact and so $A\cong B$. Taking $D=\bgmot\cong \gmot$, $G=\sgmot$ and $E=\egmot$, the left column gives a short exact sequence
\[\begin{tikzcd}
    0 \arrow{r} & \ker(\gmot\to\sgmot)\arrow{r} & \ker(\bgmot\to\egmot)\arrow{r} & \ker(\sgmot\to\egmot)\arrow{r} & 0.
\end{tikzcd}\]
\end{proof}

There is a standard conjecture due to Brown \cite{BrownDepthGraded21}, called the uneven Broadhurst-Kreimer conjecture, about the expected dimension of the subspace of depth graded motivic multiple zeta values spanned by totally odd multiple zeta values, i.e. those of the form $\zeta^\frakm(2k_1+1,\ldots,2k_r+1)$.

\begin{conjecture}\label{conj:unevenbk}
Denote by $\cA_{odd,N}$ the span of totally odd mMZVs of weight $N$. The bigraded dimensions are given by
\[\sum_{N\geq 0,r\geq 0}\dim\frac{\cD_r\cA_{odd,N}}{\cD_{r-1}\cA\cap\cD_r\cA_{odd,N}}s^Nt^d=\frac{1}{1-\cO(s)t+\cS(s)t^2},\]
where
\[\cO(s)=\frac{s^3}{1-s^2}\qquad \cS(s)=\frac{s^{12}}{(1-s^4)(1-s^6)}.\]
\end{conjecture}

Brown arrives at this conjecture from a series of observations that suggest that $\sgmot$ is the dual Lie algebra to the coalgebra spanned by depth graded totally odd motivic multiple zeta values. Thus, the above conjecture becomes a consequence of \cref{conj:homologicalbk} and the following conjecture.

\begin{conjecture}\label{conj:evendeterminesdepth}
    Denote by $\frake:=\ker(\dgmot\to\sgmot)$ the ideal generated by exceptional generators. Projection to the totally even part kills $\frake$ and defines an isomorphism $\sgmot\cong \egmot$:
    \[\ker\left(\pi_e:\dgmot\to\bigoplus_{n\geq 1}\bbQ[z_0^2,\ldots,z_n^2]\right) = \frake.\]
\end{conjecture}

On the side of motivic multiple zeta values, one may reformulate the isomorphism of this conjecture as the following.
\begin{conjecture}\label{conj:oddspansdepthmodbloc}
    Totally odd zeta values of depth $r$ span the vector space
    \[\frac{\cD_r\cL}{\cD_r\cL\cap\cB_{r-1}\cL}.\]
\end{conjecture}

\begin{proposition}\label{prop:conj1equivconj2}
    \cref{conj:oddspansdepthmodbloc} is equivalent to $\pi_e:\sgmot\to\egmot$ being an isomorphism.
\end{proposition}
\begin{proof}
    Let us first assume $\pi_e:\sgmot\to\egmot$ is an isomorphism, and fix a weight $N$ and depth $r$. We will consider all calculations in $\cD_r\cL/\cD_{r-1}\cL$ and suppose $R$ is a motivic multiple zeta value of weight $N$ and depth $r$ that is not in the span of totally odd multiple zeta values of weight $N$ and depth $r$. From the assumption that $\sgmot\cong\egmot$, we can choose a basis of totally odd mMZVs dual to the weight $N$, depth $r$ component of $\sgmot$. Via a Gram-Schmidt process, we can therefore construct an element $R_{odd}$ of the span of totally odd mMZVs such that $(R-R_{odd},\xi)=0$ for all $\xi\in\sgmot$. 

    Similarly, choosing a basis of mMZVs dual to the weight $N$, depth $r$ component of $\frake$, we can construct an linear combination $R_e$ of depth $r$ mMZVs such that $(R-R_{odd}-R_{e},\xi)=0$ for all $\xi\in\dgmot$, and hence
    \[R-R_{odd} - R_e \in \cD_{r-1}\cL\]

    Noting that all elements of $\frake=\ker(\dgmot\to\sgmot)$ of depth $r$ are necessarily of Lie length strictly less than $r$: elements of the kernel are precisely linear combinations of Lie length $n$ whose depth $n$ component vanishes.  Thus 
    \[\cD^r\frake \subset \cB^{r-1}\frake \subset \cB^{r-1}\gmot,\]
    and so $R_e\in\cB_{r-1}\cL$. As $\cD_{r-1}\cL\subset\cB_{r-1}\cL$, we therefore have
    \[R-R_{odd}\in \cB_{r-1}\cL\cap\cD_r\cL.\]

    Conversely, suppose \cref{conj:oddspansdepthmodbloc} holds, and suppose we have a $\xi\in\sgmot$ of weight $N$ and depth $r$ such that $\pi_e$ vanishes. This implies that $(R,\xi)=0$ for all totally odd $R$ and hence for all 
    \[R\in\frac{\cD_r\cL}{\cD_r\cL\cap\cB_{r-1}\cL}.\]
    Thus, if there exists a multiple zeta value $Z$ of depth $r$ such that $(Z,\xi)\neq 0$, we must have that
    \[Z\in\frac{\cD_r\cL\cap\cB_{r-1}\cL}{\cD_{r-1}\cL\cap\cB_{r-1}\cL}.\]
    In particular, $Z\in\cB_{r-1}\cL$. Thus $\xi\in\cB^{r-1}\dgmot$, and so $\xi$ is of Lie length at most $r-1$ and depth $r$. Therefore $\xi\in\frake$. Therefore $\pi_e:\sgmot\to\egmot$ is injective and hence an isomorphism.

\end{proof}

\begin{remark}\label{rem:relationsbeingodd}
    In \cite{BrownDepthGraded21}, Brown constructs candidates for generators of $\frake$ in the depth graded Lie algebra associated to the double shuffle equations, which he notes are in the kernel of $\pi_e$. If we assume, as is expected, that Brown's exceptional elements are motivic and generate $\frake$, then \cref{conj:evendeterminesdepth} would imply that $\sgmot$ is dual to the Lie coalgebra of totally odd mMZVs, from which we can conclude \cref{conj:unevenbk}
\end{remark}
\subsection{Transference of identities}\label{subsec:identities}
We may exploit concrete maps of \cref{thm:easycommutativediagram} to construct explicit relationships among motivic multiple zeta values. We first need two simple lemmas.

The first follows immediately from the definitions of the pairing between motivic multiple zeta values and the motivic Lie algebra $\gmot$.
\begin{lemma}\label{lem:depthextraction}
    Let $\xi_r$ denote the depth $r$ part of $\,\xi\in\gmot$ and denote by $p_{\xi,r}(z_0,\ldots,z_r)$ its image under the map
    \[e_0^{k_0}e_1\ldots e_1e_0^{k_r}\mapsto z_0^{k_0}\ldots z_r^{k_r}.\]
    Then $(\I^\frakl(0;\{0\}^{k_0},1,\ldots,1,\{0\}^{k_r};1),\xi)$ is equal to the coefficient of $z_0^{k_0}\ldots z_r^{k_r}$ in $p_{\xi,r}$.
\end{lemma}

The second is slightly more involved, but is a straightforward computation. We note here, that as the block filtration coincides with the coradical filtration on $\cL$, the dual block filtration coincides with the filtration by Lie length on $\gmot$. In particular, any $\xi$ of Lie length $r$ has lowest block degree terms of block degree $r$.
\begin{lemma}\label{lem:blockextraction}
    Let $w$ be a word of block degree $r$ and block decomposition $(\ell_0,\ldots,\ell_r)$. Let $\xi\in \gmot$ be of Lie length $r$,  and denote by $q_{\xi,r}(z_0,\ldots,z_r)$ the image of its block degree $r$ part in $\rbgmot$. Denote by $q_{\xi,a_0,\ldots,a_r}$ the coefficient of $z_0^{a_0}\ldots z_r^{a_r}$ in $q_{\xi,r}$. Then $(\I^\frakl(\ell_0,\ldots,\ell_r),\xi)$ is equal to 
    \[ q_{\xi,\ell_0-2,\ell_1-1,\ldots,\ell_r-1} - q_{\xi,\ell_0-1,\ldots,\ell_{r-1}-1,\ell_r-2}.\]
\end{lemma}
\begin{proof}
    From the constructions of \cref{subsec:bgLie}, $(\I^\frakl(\ell_0,\ldots,\ell_r),\xi)$ is given by the coefficient of $z_0^{\ell_0}\ldots z_r^{\ell_r}$ in the polynomial image of the block degree $r$ part of $\xi$. As noted previously, this polynomial will be equal to
    \[z_0z_1\ldots z_r(z_0-z_r)q_{\xi,r}(z_0,\ldots,z_r)\]
    from which the claim follows.
\end{proof}

\begin{definition}\label{def:depthfunctional}
    Given a depth $r$ motivic multiple zeta values $R=\I^\frakl(0;\{0\}^{k_0},1,\ldots,1,\{0\}^{k_r};1)$, define the linear functional
   \[L^{\cD}_R:\bbQ[z_0,z_1,\ldots,z_d]\to\bbQ\]
   as the linear functional extracting the coefficient of $z_0^{k_0}\ldots z_r^{k_r}$. We extend this by linearity to linear combinations of depth $r$ motivic multiple zeta values. We call the linear functional $L^{\cD}_R$ totally even if
   \[L^{\cD}_R=L^{\cD}_R\circ\pi_e.\]
\end{definition}

\begin{definition}
    Given a block degree $r$ motivic multiple zeta values $R=\I^\frakl(\ell_0,\ell_1,\ldots,\ell_r)$, define the linear functional
   \[L^{\cB}_R:\bbQ[z_0,z_1,\ldots,z_d]\to\bbQ\]
   by
   \[L^{\cB}_R\left(\sum_{a_0,\ldots,a_r\geq 0}c_{a_0,\ldots,a_r}z_0^{a_0}\ldots z_r^{a_r}\right) = c_{\ell_0-2,\ell_1-1,\ldots,\ell_r-1}-c_{\ell_0-1,\ldots,\ell_{r-1}-1,\ell_r-2}.\]
   We extend this by linearity to linear combinations of block degree $r$ motivic multiple zeta values. We call the linear functional $L^{\cB}_R$ totally even if
   \[L^{\cB}_R=L^{\cB}_R\circ\pi_e.\]
\end{definition}

\begin{theorem}\label{thm:evenfunctionalsgiverelations}
    Let $R_D$ be a linear combination of depth $r$ motivic multiple zeta values, and $R_B$ a linear combination of of block degree $r$ motivic multiple zeta values. Suppose $L^\cD_{R_D}$ and $L^\cB_{R_B}$ are both totally even and furthermore $L^\cD_{R_D}=L^\cB_{R_B}$. Then 
    \[R_B = 2^rR_D\]
    modulo products and terms of lower block degree.
\end{theorem}
\begin{proof}
    Recall that motivic iterated integrals modulo products are dual to $\gmot$. As such, a linear combination $R$ of mMZVs vanishes if and only if $(R,\xi)=0$ for all elements $\xi\in\gmot$. From the compatiblity of the Ihara bracket with depth \cite{DG05MixedTate} and block degree \cite{KeilthyBlock21}, the Lie length of an element $\xi\in\gmot$ gives a lower bound on both the depth and block degree of $\xi$. That is to say that for $\xi$ of Lie length $n$, 
    \[(R,\xi)=0\]
    for all mMZVs $R$ of depth or block degree less than $n$. As such
    \[(R_B-2^rR_D,\xi)=0\]
    for all $\xi$ of Lie length greater than $r$.
    Consider $\xi$ of Lie length exactly $r$. There exists some Lie polynomial $P$ of degree $r$ and $N>1$ such that
    \[\xi=P(\sigma_3,\sigma_5,\ldots,\sigma_{2N+1}).\]
    Since the block degree $r$ part of $\xi$ is given by the same Lie polynomial applied to the block degree 1 parts of $\{\sigma_{2k+1}\}_{k\geq 1}$, we must have
    \[(R_B,\xi)=(R_B,P(s_3,\ldots,s_{2N+1}))=L^\cB_{R_B}\left(P(p_3,\ldots,p_{2N+1})(z_0,\ldots,z_r)\right).\]
    As $L^\cB_{R_B}$ is a totally even functional, and $\pi_e$ is a Lie algebra homomorphism, we have
    \[L^\cB_{R_B}\left(P(p_3,\ldots,p_{2N+1})\right)=L^\cB_{R_B}\left(\pi_e(P(p_3,\ldots,p_{2N+1}))\right)=L^\cB_{R_B}\left(P(e_3,\ldots,e_{2N+1})\right).\]
    Similarly, recalling that $\pi_e\left(\phi_{2k+1}\right)=\frac{1}{2}e_{2k+1}$, we must have
    \begin{align*}
      (R_D,\xi)  =&{} L^\cD_{R_D}\left(P(\phi_3,\ldots,\phi_{2N+1})\right)\\
        =&{} L^\cD_{R_D}\left(\pi_e(P(\phi_3,\ldots,\phi_{2N+1}))\right)\\
        =&{} \frac{1}{2^r}L^\cD_{R_D}\left(P(e_3,\ldots,e_{2N+1})\right).
    \end{align*}
    As $L^\cD_{R_D}=L^\cB_{R_B}$, we therefore have that 
    \[(R_B-2^rR_D,\xi) = (L^\cB_{R_B}-L^\cD_{R_D})\left(P(e_3,\ldots,e_{2N+1})\right)=0.\]

    Since the block filtration on $\cL$ is dual to the filtration by Lie length on $\gmot$, we can construct a dual basis to $\cB^{r-1}\gmot$ in $\cB_{r-1}\cL$, and so by applying a Gram-Schmidt process, we may construct an element $R_{corr}\in\cB_{r-1}\cL$ such that $(R_B-2^rR_D+R_{corr},\xi)=0$ for all $\xi$ of Lie length less than $r$. Thus, we must have that
    \[R_B-2^rR_D = 0\]
    modulo products and terms of lower block degree.
\end{proof}

\begin{remark}\label{rem:parityblockdrop}
    Due to a block parity result noted first by Charlton \cite{CharltonBlock21,KeilthyBlock21}, for $R_B$ and $R_D$ satisfying the conditions of the above theorem, we in fact have that
    \[ R_B-2^rR_D = 0 \text{ modulo }\cB_{r-2}\cL.\]    
\end{remark}

As an example, we have the following easy corollary.

\begin{corollary}
    For any $0\leq 2a\leq n$, we have that
    \[ (-1)^{n+1}\sum_{k=a}^{n-a} \zeta^\frakl(\{2\}^k,6,\{2\}^{n-k}) = 16\zeta^\frakl(1,1,2n-2a+3,2a+1)\]
    modulo products and block degree three.
\end{corollary}
\begin{proof}
    The left hand side is homogeneous of block degree $4$. Using the dihedral symmetries of $\rbgmot$ \cite{KeilthyBlock21}, it is an easy calculation to see that the sum
    \[(-1)^{n+1}\sum_{k=a}^{n-a} \zeta^\frakl(\{2\}^k,6,\{2\}^{n-k})\]
    corresponds to a totally even linear functional that extracts the coefficient of $z_0^{2a}z_4^{2n-2a+2}$ in a polynomial. Similarly, using the dihedral symmetries of the depth graded Lie algebra, one can see that
    \[\zeta^\frakl(1,1,2n-2a+3,2a+1)\]
    corresponds to the same linear functional. The conditions of \cref{thm:evenfunctionalsgiverelations} hold, and therefore the result follows.
\end{proof}
\begin{remark}
    If we instead consider symmetric sums of $\zeta^\frakl(\{2\}^k,4,\{2\}^{n-k})$, we recover the results of \cite{CK22Evaluation242}. However, in that article, the authors were able to more this result more precise, determining products and terms of lower block degree. Indeed, the authors were in fact able to write individual zeta values of the form $\zeta^\frakm(\{2\}^k,4,\{2\}^{n-k})$ in terms of double zeta values. While a similar computation should be possible, given the complexity of the relation in \cite{CK22Evaluation242}, we don't expect to obtain a reasonable formula in this setting.
\end{remark}

\begin{corollary}
    If $R_B$ is a linear combination of block degree $r$ defines a totally even functional $L^\cB_{R_B}$, then $R_B$ can be explicitly expressing in terms of totally odd multiple zeta values of depth $r$, modulo products and terms of lower block degree.
\end{corollary}
\begin{proof}
    The surjectivity of the map $\pi_e:\sgmot\to\egmot$ tells us that given such a $R_B$, we can always construct an $R_D$ of depth $r$ such that $L^\cD_{R_D}=L^\cB_{R_B}$. That $R_D$ can be expressed in terms of totally odd multiple zeta values is a result of the functional being totally even.
\end{proof}

The converse to the above corollary is well established \cite{BrownMTM12,BrownLetter16,KeilthyBlock21}. It is known that every element of $\cB_r\cA$ can be written in terms of the Hoffman basis
\[\{\zeta^\frakm(a_1,\ldots,a_k)\mid a_i\in\{2,3\}\}\]
with at most $r$ threes, i.e. of block degree $r$. As $\cD_r\cA\subset\cB_r\cA$, we must be able to all depth $r$ mMZVs in terms of block degree $r$ mMZVs modulo terms of lower block degree. However, \cref{thm:evenfunctionalsgiverelations} gives us a precise expression, in terms of ``almost'' Hoffman elements. 

\begin{corollary}\label{cor:oddtoalmosthoffman}
    For $k_1,\ldots,k_{r-1}\geq 1$,
    \[\zeta^\frakl(2k_1+1,\ldots,2k_r+1) = \frac{(-1)^{k_1+\cdots+k_r+r+1}}{2^r}\zeta^\frakl_1(\{2\}^{k_1-1},3,\ldots,3,\{2\}^{k_r})\]
    modulo terms of lower block degree.
\end{corollary}
\begin{proof}
    Applying \cref{thm:evenfunctionalsgiverelations}, 
    \[\zeta^\frakl(2k_1+1,\ldots,2k_r+1)=(-1)^r\I^\frakl(0;1,\{0\}^{2k_1},\ldots,1,\{0\}^{2k_r};1)\]
    and
    \[(-1)^{r+1}\I^\frakl(1,2k_1+1,\ldots,2k_{r-1}+1,2k_r+2)\]
    define the same totally even linear functional. Thus
    \[2^r\zeta^\frakl(2k_1+1,\ldots,2k_r+1)=(-1)^{r+1}\I^\frakl(1,2k_1+1,\ldots,2k_{r-1}+1,2k_r+2)\]
    modulo terms of lower block degree. But, untangling the notation
    \begin{align*}
        \I^\frakl(1,2k_1+1,\ldots,2k_{r-1}+1,2k_r+2) =&{} \I^\frakl(0;0,\{1,0\}^{k_1},0,\{1,0\}^{k_2},\ldots,0,\{1,0\}^{k_r};1)\\
        =&{}(-1)^{k_1+\cdots+k_r}\zeta^\frakl_1(\{2\}^{k_1-1},3,\ldots,3,\{2\}^{k_r}).
    \end{align*}
    Thus
    \[\zeta^\frakl(2k_1+1,\ldots,2k_r+1) = \frac{(-1)^{k_1+\cdots+k_r+r+1}}{2^r}\zeta^\frakl_1(\{2\}^{k_1-1},3,\ldots,3,\{2\}^{k_r}).\]
\end{proof}
\begin{remark}
    The same argument allows us to extend the above to all $k_1,\ldots,k_r\geq 0$, obtaining a divergent MZV with entries strictly greater than $1$. The general expression depends on which $k_i$ are equal to zero, but may be obtained from the above formula by suitable interpretation. Begin with expression obtained in \cref{cor:oddtoalmosthoffman}, and apply the following formal reduction rules until a legal expression is obtained, i.e. an expression containing no $\{2\}^{-1}$
        \begin{enumerate}
            \item $\zeta^\frakl_m(\{2\}^{-1},3,\{2\}^{n},\ldots) \mapsto \zeta^\frakl_{m+1}(\{2\}^n,\ldots)$,
            \item $\zeta^\frakl_m(\ldots,3+a,\{2\}^{-1},3+b,\ldots)=\zeta^\frakl_m(\ldots,4+a+b,\ldots)$.
        \end{enumerate}
    However, as it is expected that totally odd MZVs with entries strictly greater than $1$ span all totally odd MZVs (indeed, they are often defined as such), the case considered in the corollary is already interesting. Given the obvious comparison to the Hoffman basis, one wonders whether a similar analysis to that of Brown is possible \cite{BrownMTM12}.
\end{remark}

The following stronger conjecture is known to hold for $r=1,\,2$, and seems to hold numerically to high weight for $r=3,\, 4$. Thus we suggest that this might hold in full generality.

\begin{conjecture}\label{conj:evenrelationmoddepth}
    Let $R_D$ be a linear combination of depth $r$ motivic multiple zeta values, and $R_B$ a linear combination of of block degree $r$ motivic multiple zeta values. Suppose the corresponding linear functionals $L^\cD_{R_D}$ and $L^\cB_{R_B}$ are both totally even and furthermore $L^\cD_{R_D}=L^\cB_{R_B}$. Then 
    \[R_B = 2^rR_D\]
    modulo products and terms of lower \emph{depth}.
\end{conjecture}

\section{Connections to existing conjectures}\label{sec:existingconj}
In this final section, we will further explore some of the connections mentioned between the proceeding discussion and existing variations of the Broadhurst-Kreimer conjectures. In particular, we wish to mention how some of these results may be used to provide a potential strategy in establishing these conjectures.

For example, as noted in \cref{rem:relationsbeingodd}, the uneven Broadhurst Kreimer conjecture could be established in three parts
\begin{enumerate}
    \item Establish the isomorphism $\pi_e:\sgmot\cong\egmot$,
    \item Establish that $\frake\subset\ker\pi_e$,
    \item Establish that $\frake$ is generated by period polynomial relations.      
\end{enumerate}

As noted in \cref{prop:conj1equivconj2}, the first of these is equivalent to
\[\frac{\cD_R\cL}{\cD_r\cL\cap\cB_{r-1}\cL}\]
being spanned by totally odd mMZVs. A possible proof strategy for this is the following.

In \cite{BrownMTM12}, Brown introduces the infinitesimal coactions $D_k:\cA_N\to \cL_k\otimes\cA_{N-k}$. These descend by antisymmetrization to derivations on the Lie coalgebra of indecomposables 
\[\partial_k:=\left(D_k - \tau\circ D_{N_k}\right):\cL_N\to \cL_{k}\otimes\cL_{N-k},\]
where $\tau:\cL\otimes\cL\to\cL\otimes\cL$ swaps the left and right tensor components. Since the block filtration is equal to the coradical filtration, a linear combination $R\in\cL$ is an element of $\cB_{r-1}\cL$ if and only if
\[(\id^{\otimes r-1}\otimes \partial_{k_r})\circ(\id^{\otimes r-2}\otimes \partial_{k_{r-1}})\circ\cdots\circ (\partial_{k_1})(R)=0\]
for every $r$-tuple $(k_1,\ldots,k_r)$.

In fact, it is sufficient to consider only tuples of odd integers $k_i\geq 3$. Furthermore, we can compose with $(-,\sigma_{2k+1})$ in the left tensor factor
\[\bar{\partial}_{2k+1}:= \big((-,\sigma_{2k+1})\otimes \id\big)\circ \partial_r\]
to obtain derivations $\bar{\partial}_{2k+1}:\cL_N\to\cL_{N-1-2k}$. Then $R\in\cB_{r-1}\cL$ if and only if it is annihilated by every $r$-fold composition of these derivations.

From this discussion, we conclude the following.

\begin{proposition}\label{prop:spanningstrategy}
    In order to show that $\pi_e:\sgmot\cong\egmot$, it suffices to show that for every $Z\in\cD_r\cL$, there exists a linear combination $R$ of depth $r$ totally odd mMZVs such that
    \[\bar{\partial}_{2k_r+1}\circ \bar{\partial}_{2k_{r-1}+1}\circ\cdots\bar{\partial}_{2k_1+1}(Z-R)=0\]
    for every $r$-tuple $(2k_1+1,\ldots,2k_r+1)$ with $k_i\geq 1$.
\end{proposition}

With this set up, we could potentially establish the third part as well. Denote by $\bar{D}_{2k+1}$ the composition $\big((-,\sigma_{2k+1})\otimes\id\big)\circ D_{2k+1}$. Bu considering the kernel of 
\[\left(\bigoplus_{k\geq 1}\bar{D}_{2k+1}\right)^r\,\]
restricted to $\cD_r\cA_N$, it is theoretically possible to compute the dimension of
\[\frac{\cD_r\cA_N}{\cD_r\cA_N\cap\cB_{r-1}\cA_N}\]
from which it is possible to provide evidence that $\frake$ is generated as a Lie ideal by elements corresponding to the quadratic period polynomial relations.

\begin{proposition}\label{prop:quadraticgenerationstrategy}
    For each cusp form $f$, denote by $e_f$ the element of $\dgmot$ corresponding to the associated quadratic relation. Suppose that $H_i(\dgmot)=0$ for $i\geq 3$. Then the claim that $\frake$ is generated as a Lie ideal by $\{e_f\}$ is equivalent to 
    \[\sum_{N\geq 0,r\geq 0}\dim\frac{\cD_r\cA_N}{\cD_r\cA_N\cap\cB_{r-1}\cA_N}s^Nt^r = \frac{1}{1-\cO(s)t+\cS(s)t^2}\]
    where 
    \[\cO(s)=\frac{s^3}{1-s^2}\text{, and }\cS(s)=\frac{s^{12}}{(1-s^4)(1-s^6)}.\]
\end{proposition}
\begin{proof}
    As established in \cref{prop:conj1equivconj2}, 
    \[\dim\frac{\cD_r\cA_N}{\cD_r\cA_N\cap\cB_{r-1}\cA_N}\]
    is determined by the dimensions of $\sgmot$. As we have assumed $H_i(\dgmot)=0$ for $i\geq 3$, we can write down the associated generating series. Specifically, if $A(s)$ is the generating series of the weight graded dimensions of the set of generators of $\sgmot$, and $B_r(s)$ is the generating series of the weight graded dimensions of the Lie length $r$ part of $\ker(\gmot\to\sgmot)$, equivalently of $\frake$, then
    \[\sum_{N\geq 0,r\geq 0}\dim\frac{\cD_r\cA_N}{\cD_r\cA_N\cap\cB_{r-1}\cA_N}s^Nt^r = \frac{1}{1-A(s)t+\sum_{r\geq 1}B_r(s)t^r}.\]
    In our case, we have $A(s)=\cO(s)$, $B_1(s)=0$, and $B_2(s)=\cS(s)$ \cite{Pollack09Thesis}. Thus, the claim follows.
\end{proof}
%
%The results of \cref{subsec:identities} also provide another perspective on the uneven Broadhurst Kreimer conjecture.
%
%If we assume $\pi_e(\frake)=0$, and denote by $\cL_{odd,N}$  the weight $N$ span of totally odd mMZVs in $\cL$, then
%\[\dim\frac{\cD_r\cL_{odd,N}}{\cD_{r-1}\cL_N\cap\cD_r\cL_{odd,N}} = \dim\frac{\cD_r\cL_{odd,N}}{\cD_r\cL_{odd,N}\cap\cB_{r-1}\cL_N}.\]
%To see this, note that the vector spaces of the right hand side are dual to $\egmot$, while the vector spaces of left hand side are dual to $\pi_e(\dgmot)$, which is equal to $\egmot$ by our assumption.
%
%From \cref{cor:oddtoalmosthoffman}, it may be possible to analyse this via these ``almost'' Hoffman elements. Denote by $\cL^{(H)}_{r,N}$ the weight $N$ subspace of $\cL$ spanned by 
%\[\{ \zeta^\frakl_1(\{2\}^{k_1-1},3,\ldots,3,\{2\}^{k_r})\}.\]
%We then have that
%\[\dim\frac{\cD_r\cL_{odd,N}}{\cD_r\cL_{odd,N}\cap\cB_{r-1}\cL_N} = \dim\frac{\cL^{(H)}_{r,N}}{\cL^{(H)}_{r,N}\cap\cB_{r-1}\cL_N}.\]
%If we instead assume \cref{conj:evenrelationmoddepth}, then we would instead have that
%\[\dim\frac{\cD_r\cL_{odd,N}}{\cD_{r-1}\cL_N\cap\cD_r\cL_{odd,N}}\]
%is equal to the dimension
%\[\frac{\cL^{(H)}_{r,N}}{\cD_{r-1}\cL\cap\cL^{(H)}_{r,N}}.\]
%Given the obvious comparison to the Hoffman basis, one would hope that an analysis similar to Brown's should be possible. 

%\input{results.tex}
%%%%%%%%%%%%%%%%%%%%%%%%%%%%%%%%%%%%%%%%%%%%%%%%%%
%%% BIBLIOGRAPHY

\vspace{1.5\baselineskip}
\renewcommand{\baselinestretch}{.8}
\Needspace*{4em}
\printbibliography%[heading=bibnumbered]

\Needspace*{3\baselineskip}
\noindent
\rule{\textwidth}{0.15em}

{\noindent\small
Chalmers tekniska högskola och Göteborgs Universitet,
Institutionen för Matematiska vetenskaper,
SE-412 96 Göteborg, Sweden\\
E-mail: \url{keilthy@chalmers.se}
}\vspace{.5\baselineskip}
\end{document}

%% file: packages.tex
\usepackage{etoolbox}
\usepackage{ifdraft}
\usepackage{iftex}

\ifluatex
\usepackage[utf8]{luainputenc}
\usepackage[T1]{fontenc}
\fi

\usepackage[english]{babel}
\usepackage{csquotes}

\usepackage{lmodern}
\usepackage{fourier}

\usepackage{amsfonts}
\usepackage{mathrsfs} % script characters (\mathscr)
\usepackage{dsfont}
\usepackage{shuffle}

\usepackage{amssymb}
\usepackage{amsmath}
\usepackage{mathtools}
\usepackage[thmmarks, amsmath, amsthm]{ntheorem}
\usepackage{nccmath}    % !include after amsmath!

\usepackage[draft=false]{scrlayer-scrpage}

\usepackage{needspace}

\providebool{nobiblatex}
\boolfalse{nobiblatex}
\ifbool{nobiblatex}{%
}{%
  \usepackage[backend=biber,style=numeric-comp,giveninits,url=false,sortcites=true,maxbibnames=99]{biblatex}
  \addbibresource{bibliography.bib}
}

\usepackage[final,
            pdfborder={0 0 0}, colorlinks=true,
            linkcolor=BrickRed, citecolor=ForestGreen, urlcolor=RoyalBlue]{hyperref}
\usepackage[noabbrev,capitalize]{cleveref}

\usepackage[cmyk,dvipsnames]{xcolor}

\ifdraft{%
\usepackage[textsize=scriptsize,colorinlistoftodos]{todonotes}
\usepackage{lineno}
\usepackage{scrtime}
}{}

\usepackage{booktabs}
\usepackage[inline]{enumitem}
\usepackage{lettrine}
\usepackage{tikz-cd}
\usepackage{url}

%% file: layout.tex
\ifdraft{\date{\today}}{\date{}}

\KOMAoptions{DIV=12}

\ifdraft{%
\linenumbers
\newcommand*\patchAmsMathEnvironmentForLineno[1]{%
  \expandafter\let\csname old#1\expandafter\endcsname\csname #1\endcsname
  \expandafter\let\csname oldend#1\expandafter\endcsname\csname end#1\endcsname
  \renewenvironment{#1}%
     {\linenomath\csname old#1\endcsname}%
     {\csname oldend#1\endcsname\endlinenomath}}%
\newcommand*\patchBothAmsMathEnvironmentsForLineno[1]{%
  \patchAmsMathEnvironmentForLineno{#1}%
  \patchAmsMathEnvironmentForLineno{#1*}}%
\AtBeginDocument{%
\patchBothAmsMathEnvironmentsForLineno{equation}%
\patchBothAmsMathEnvironmentsForLineno{align}%
\patchBothAmsMathEnvironmentsForLineno{flalign}%
\patchBothAmsMathEnvironmentsForLineno{alignat}%
\patchBothAmsMathEnvironmentsForLineno{gather}%
\patchBothAmsMathEnvironmentsForLineno{multline}%
}}

\clearmainofpairofpagestyles
\automark[section]{subsection}

\renewcommand{\subsectionmark}[1]{}

\lohead{{\small\normalfont \headertitle}}
\rohead{{\small \headerauthors}}

\RedeclareSectionCommand[%
  font=\Large\sffamily\bfseries,%
  beforeskip=1\baselineskip,%
  afterskip=0.5\baselineskip,%
  indent=0em,%
  afterindent=false%
  ]{section}

\RedeclareSectionCommands[%
  font=\normalfont\bfseries,%
  beforeskip=3pt,%
  afterskip=-1em%
  ]{subsection,subsubsection}

\RedeclareSectionCommands[%
  font=\normalfont\itshape,%
  beforeskip=1pt,%
  afterskip=-1em,%
  indent=1em%
  ]{paragraph}

\ifbool{nobiblatex}{}{%
  \AtEveryBibitem{\clearfield{doi}}
  \AtEveryBibitem{\clearfield{isbn}}
  \AtEveryBibitem{\clearfield{issn}}
  \AtEveryBibitem{\clearfield{pages}}
  \AtEveryBibitem{\clearlist{language}}

  \setlength\bibitemsep{3pt}
  
  \renewbibmacro*{in:}{}
  \DeclareFieldFormat
    [article,inbook,incollection,inproceedings,patent,thesis,unpublished]
    {title}{#1}
}

\newenvironment{enumeratearabic*}{
\begin{enumerate*}[label=(\arabic*)] %%, leftmargin=0pt,labelindent=2em,itemindent=!]
}{
\end{enumerate*}
}

\newenvironment{enumerateroman*}{
\begin{enumerate*}[label=(\roman*)] %%, leftmargin=0pt,labelindent=2em,itemindent=!]
}{
\end{enumerate*}
}

\numberwithin{equation}{section}

\theoremnumbering{arabic}
\newtheorem{theoremcounter}{theoremcounter}[section]
\theoremnumbering{Alph}

\theoremstyle{plain}

\newtheorem{conjecture}[theoremcounter]{Conjecture}

\newtheorem{corollary}[theoremcounter]{Corollary}

\newtheorem{lemma}[theoremcounter]{Lemma}

\newtheorem{proposition}[theoremcounter]{Proposition}

\newtheorem{theorem}[theoremcounter]{Theorem}

\theoremstyle{plain}

\theoremstyle{definition}

\newtheorem{definition}[theoremcounter]{Definition}

\newtheorem{example}[theoremcounter]{Example}

\theoremstyle{remark}

\newtheorem{remark}[theoremcounter]{Remark}

\theoremstyle{nonumberremark}

%% file: shortcuts.tex
\newcommand{\minwidthmathbox}[2]{%
  \mathmakebox[{\ifdim#1<\width\width\else#1\fi}]{#2}%
}

\newcommand{\bbC}{\ensuremath{\mathbb{C}}}

\newcommand{\bbP}{\ensuremath{\mathbb{P}}}
\newcommand{\bbQ}{\ensuremath{\mathbb{Q}}}

\newcommand{\bbZ}{\ensuremath{\mathbb{Z}}}

\newcommand{\bbone}{\ensuremath{\mathds{1}}}

\newcommand{\cA}{\ensuremath{\mathcal{A}}}
\newcommand{\cB}{\ensuremath{\mathcal{B}}}
\newcommand{\cC}{\ensuremath{\mathcal{C}}}
\newcommand{\cD}{\ensuremath{\mathcal{D}}}

\newcommand{\cH}{\ensuremath{\mathcal{H}}}

\newcommand{\cL}{\ensuremath{\mathcal{L}}}

\newcommand{\cO}{\ensuremath{\mathcal{O}}}

\newcommand{\cS}{\ensuremath{\mathcal{S}}}

\newcommand{\fraka}{\ensuremath{\mathfrak{a}}}
\newcommand{\frakb}{\ensuremath{\mathfrak{b}}}

\newcommand{\frakd}{\ensuremath{\mathfrak{d}}}
\newcommand{\frake}{\ensuremath{\mathfrak{e}}}

\newcommand{\frakg}{\ensuremath{\mathfrak{g}}}

\newcommand{\frakl}{\ensuremath{\mathfrak{l}}}
\newcommand{\frakm}{\ensuremath{\mathfrak{m}}}

\newcommand{\frakr}{\ensuremath{\mathfrak{r}}}
\newcommand{\fraks}{\ensuremath{\mathfrak{s}}}

\newcommand{\frakA}{\ensuremath{\mathfrak{A}}}

\NewCommandCopy{\rightarroworig}{\rightarrow}
\renewcommand{\rightarrow}
  {\DOTSB\protect\relbar\mspace{-9.7mu}\rightarroworig}

\renewcommand{\twoheadrightarrow}
  {\DOTSB\protect\rightarroworig\mspace{-15mu}\rightarroworig}

\NewCommandCopy{\leftarroworig}{\leftarrow}
\renewcommand{\leftarrow}
  {\DOTSB\protect\leftarroworig\mspace{-9.7mu}\relbar}

\newcommand{\id}{\ensuremath{\mathrm{id}}}

%% file: shortcuts_this.tex
\DeclareMathOperator{\Spec}{Spec}

\DeclareMathOperator{\Lie}{Lie}

\DeclareMathOperator{\I}{I}

\DeclareMathOperator{\per}{per}
\DeclareMathOperator{\Sh}{Sh}

\def\sh{\shuffle}
\def\degb{{\deg_\mathcal{B}}}

\def\gmot{\frakg^\frakm}
\def\dgmot{\frakd\frakg}
\def\sgmot{\fraks\frakd\frakg}
\def\bgmot{\frakb\frakg}
\def\rbgmot{\frakr\frakb\frakg}
\def\egmot{\frake\frakg}